\documentclass[11pt]{article}
\usepackage[ansinew]{inputenc}
\usepackage{graphicx}
\usepackage{color}
\usepackage{enumerate,latexsym}
\usepackage{latexsym}
\usepackage{amsmath,amssymb}
\usepackage{graphicx}

\usepackage{amsthm}
\def\cC{\mathcal{C}}
\def\cL{\mathcal{L}}
\newfont{\bb}{msbm10 at 11pt}
\newfont{\bbsmall}{msbm8 at 8pt}
\def\cS{\mathcal{S}}

\def\S{{\Sigma}}

\def\cF{\mathcal{F}}
\newcommand{\ov}{\overline}
\def\rth{\mathbb{R}^3}
\def\R{\mathbb{R}}
\def\B{\mathbb{B}}
\def\oB{\ov{\mathbb{B}}}
\def\N{\mathbb{N}}

\def\be{\beta}

\newcommand{\ben}{\begin{enumerate}}

\newcommand{\bit}{\begin{itemize}}
\newcommand{\een}{\end{enumerate}}
\newcommand{\eit}{\end{itemize}}
\newcommand{\wh}{\widehat}
\newcommand{\Int}{\mbox{Int}}
\newcommand{\cH}{{\mathcal H}}

\newcommand{\wt}{\widetilde}

\newcommand{\ed}{\end{document}}

\def\a{{\alpha}}

\def\t{{\theta}}

\def\g{{\gamma}}
\def\G{{\Gamma}}

\def\L{{\Lambda}}
\def\de{{\delta}}

\def\ve{{\varepsilon}}

\newcommand{\cD}{{\mathcal D}}

\newtheorem{theorem}{Theorem}[section]

\newtheorem{remark}[theorem]{Remark}
\newtheorem{corollary}[theorem]{Corollary}
\newtheorem{definition}[theorem]{Definition}

\newtheorem{claim}[theorem]{Claim}

\definecolor{b}{rgb}{.1,.1,.7}
\definecolor{rr}{rgb}{.8,0,.3}
\definecolor{g}{rgb}{0,.5,0}
\definecolor{pp}{rgb}{.5,0,.7}
\definecolor{r}{rgb}{.6,0,.3}
\definecolor{y}{rgb}{.9,.99,.9}

\begin{document}
\begin{title}{Limit lamination theorem for H-disks}
\end{title}



\date{}

\begin{author}
{William H. Meeks III\thanks{This material is based upon
   work for the NSF under Award No. DMS-1309236.
   Any opinions, findings, and conclusions or recommendations
   expressed in this publication are those of the authors and do not
   necessarily reflect the views of the NSF.}
   \and Giuseppe Tinaglia\thanks{The second author was partially
   supported by
EPSRC grant no. EP/M024512/1}}
\end{author}
\maketitle

\begin{abstract}
In this paper we prove a  theorem concerning  lamination limits of
sequences of compact
disks $M_n$ embedded in $\rth$ with constant mean curvature $H_n$,
when the boundaries of these disks tend to
infinity.  This theorem generalizes
to the non-zero constant mean curvature case Theorem~0.1 by Colding
and Minicozzi~\cite{cm23}. We apply
this  theorem to prove the existence of a  chord arc result for
compact
disks  embedded in $\rth$ with constant mean curvature; this
chord arc result generalizes
Theorem~0.5 by  Colding and Minicozzi in~\cite{cm35} for minimal disks.

\vspace{.3cm}


\end{abstract}
\maketitle

\vspace{-.5cm}

\section{Introduction}
In this paper  we apply results in~\cite{mt8,mt7,mt9} to obtain
(after passing to a subsequence)
minimal  lamination limits for any
sequence of compact
disks $M_n$ embedded in $\rth$ with constant mean curvature $H_n$,
when the boundaries of these disks tend to infinity; see
Theorem~\ref{thm2.1} below.  This  theorem is inspired by and generalizes
to the non-zero constant mean curvature setting
Theorem~0.1 by Colding and Minicozzi~\cite{cm24} and
is  related to work in~\cite{bt1,me25,me30,wh15}.
As an application of
Theorem~\ref{thm2.1}, we obtain a chord arc result for
compact  disks  embedded in $\rth$ with constant mean curvature
that does not depend on the value of the mean curvature;
this chord arc result is stated below in Theorem~\ref{main2}.
Theorem~\ref{main2} is inspired by and generalizes the chord arc
result by  Colding and Minicozzi for embedded minimal disks appearing  in
Theorem~0.5 of~\cite{cm35}.

For clarity of  exposition, we will call an oriented surface
$M$ immersed in $\rth$ an {\it $H$-surface} if it
is {\it embedded}, {\em connected}  and it
has {\it non-negative constant mean curvature $H$}. We will  call an
$H$-surface an {\em $H$-disk} if the $H$-surface is homeomorphic
to a closed unit disk in the Euclidean plane.
Here  $\B(R)$ denotes the open ball in $\rth$ centered at the origin $\vec{0}$
of radius $R$ and for a
point $p$ on a surface $\Sigma \subset \rth$, $|A_{\Sigma}|(p)$ denotes the norm
of the second fundamental
form of $\Sigma$ at $p$.

\begin{theorem}[Limit lamination theorem for $H$-disks] \label{thm2.1}
Fix  $\ve >0$ and let $\{M_n\}_n$ be
a sequence of $H_n$-disks in $\rth$ containing the origin  and such that
$\partial M_n \subset [\rth - \B(n)]$ and
$|A_{M_n} |(\vec{0})\geq \ve$. Then, after replacing  by
some subsequence,  exactly one of the
following two statements hold.
\ben[A.]
\item The surfaces $M_n$ 
converge smoothly with multiplicity one or two on compact
subsets of $\rth$ to a helicoid $M_{\infty}$
containing the origin. Furthermore,
every component $\Delta$ of $M_n\cap \B(1)$  is an open disk
whose closure $\ov{\Delta}$ in $M_n$
is a compact disk with piecewise smooth boundary, and
where the intrinsic distance in $M_n$
between  any two points in  $\ov{\Delta}$ is less than 10.
\item  There are points $p_n\in M_n$ such that
\[
\lim_{n\to \infty}p_n=\vec{0} \text{\, and \,}
\lim_{n\to \infty}|A_{M_n}|(p_n)=\infty,
\] and the
following  hold: \ben
\item The surfaces $M_n$ converge to a foliation
of $\rth$ by planes and the convergence is $C^\alpha$, for any $\a\in(0,1)$,
away from the line containing
the origin and orthogonal to the planes in the foliation. \item
There exists  compact subdomains
$\cC_n$ of $M_n$, $[M_n\cap \oB(1)]\subset \cC_n \subset \B(2)$
and  $\partial \cC_n\subset \B(2)-\oB(1)$, each $\cC_n$
consisting of one or two pairwise disjoint disks, where
each disk component has intrinsic diameter less than 3 and intersects $\B(1/n)$.
Moreover, each connected
component of $M_n\cap \B(1)$ is an open disk whose closure in $M_n$
is a compact disk with piecewise smooth boundary. \een
\een
\end{theorem}

As an application of Theorem~\ref{thm2.1} and its proof, we prove
in Theorem~\ref{main3} a sharp chord arc result for $H$-disks (also see
Remark~\ref{sharp});
the following
chord arc result for $H$-disks is a restatement of equation~\eqref{eqChord2}
in Theorem~\ref{main3}.

\begin{theorem}[Chord arc property  for $H$-disks] \label{main2}
There exists a $C>1$ so that the following holds.
Suppose that $\S$ is an $H$-disk, $ \vec{0}\in \S$ and $R>r_0>0$.
If the open intrinsic ball ${B}_\S(\vec 0,CR)$
of radius $CR$ centered at $\vec{0}$ is
 contained in $\S-\partial \S$ and
$\sup_{B_\S(\vec 0,(1-\frac{\sqrt 2}2)r_0)}|A_\S|>r_0^{-1}$,
then
$$ \frac{1}{3}\mbox{\rm dist}_\S (x,\vec{0})\leq |x|/2+r_0, \;
\mbox{\rm for } x\in B_\S(\vec 0,R).$$
\end{theorem}


The proofs of the results described in this
paper depend in an essential manner on the
existence of
extrinsic curvature estimates for disks embedded
in $\rth$ of non-zero constant mean curvature
that appear in~\cite{mt7}, as well as on a key  extrinsic one-sided curvature
estimate obtained in~\cite{mt9} and a weak cord arc result derived in~\cite{mt8};
these results from~\cite{mt8,mt7,mt9}  are described in  Section~\ref{sec:pre}.



%

\section{Preliminaries.} \label{sec:pre}

Throughout this paper, we use the following notation.
Given $a,b,R>0$, $p\in \rth$ and $\S$ a surface in $\rth$:

\bit
\item $\B(p,R)$ is the open ball of radius $R$ centered at $p$.
\item $\B(R)=\B(\vec{0},R)$, where $\vec{0}=(0,0,0)$.
\item For $p\in \S$, $B_{\S}(p,R)$ denotes the open intrinsic ball in $\S$ of radius $R$.
\item $C(a,b)=\{(x_1,x_2,x_3) \mid x_1^2+x_2^2\leq a^2,\, |x_3|\leq b\}$.
\item $A(r_1,r_2)=\{(x_1,x_2,0)\mid r_2^2\leq x_1^2+x_2^2\leq r_1^2\}$.
\eit


We first  introduce the notion of multi-valued
graph, see~\cite{cm22} for further discussion.
Intuitively, an $N$-valued graph is a simply-connected embedded surface covering an
annulus such that over a neighborhood of each point of the annulus, the
surface consists of $N$ graphs. The stereotypical  infinite multi-valued
graph is half of the helicoid, i.e., half of an infinite  double-spiral staircase.

\begin{figure}[h]
\begin{center}
\includegraphics[width=11.5cm]{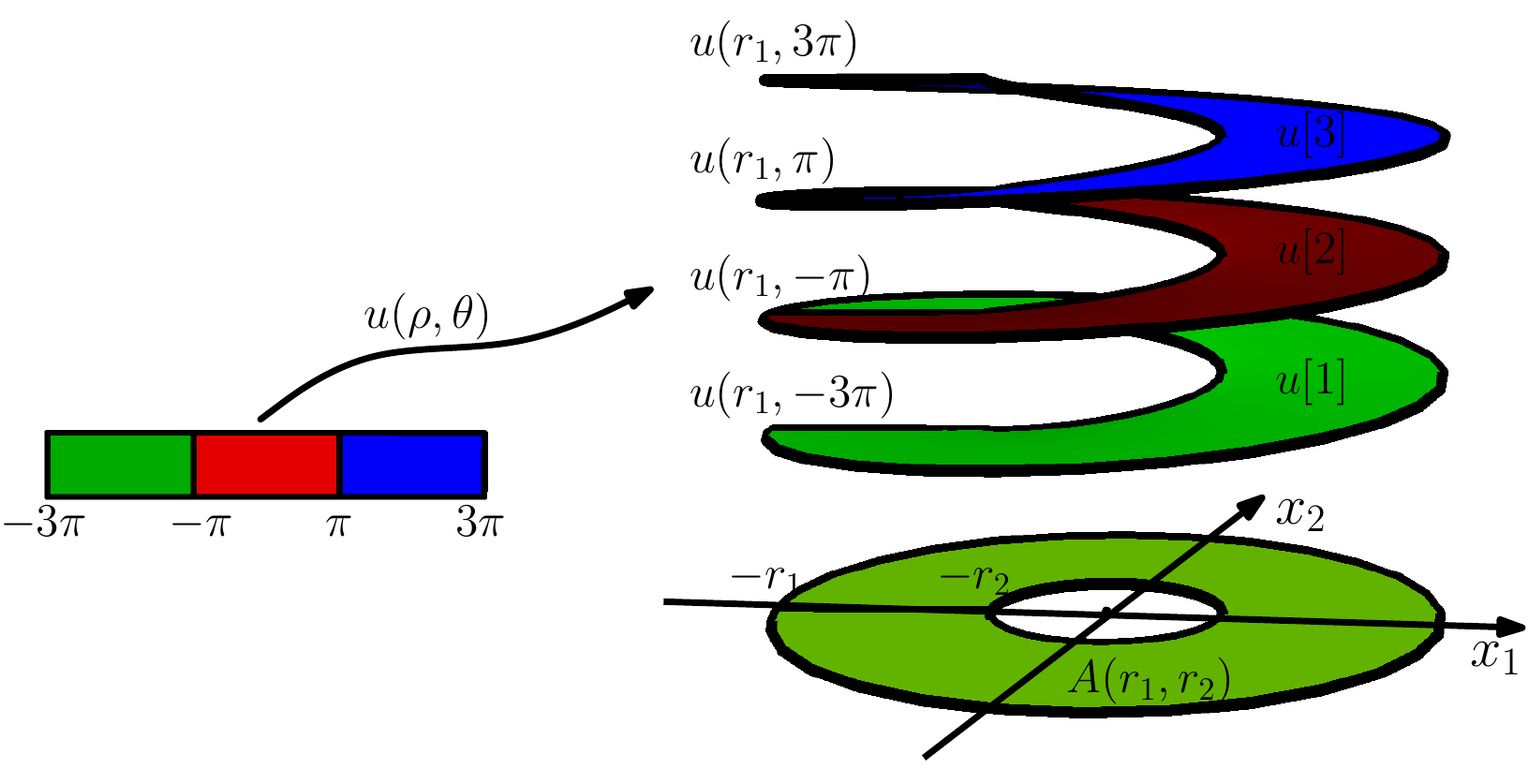}
\caption{A right-handed 3-valued graph.}
\label{3-valuedgraph}
\end{center}
\end{figure}

\begin{definition}[Multi-valued graph]\label{multigraph} {\rm
Let $\mathcal{P}$ denote the universal cover of the
punctured $(x_1,x_2)$-plane,
$\{(x_1,x_2,0)\mid (x_1,x_2)\neq (0,0)\}$, with global coordinates
$(\rho , \theta)$.
\ben[1.] \item
An {\em $N$-valued graph over the annulus $ A(r_1,r_2)$} is a single valued graph
$u(\rho, \theta)$ over $\{(\rho ,\theta )\mid r_2\leq \rho
\leq r_1,\;|\theta |\leq N\pi \}\subset \mathcal{P}$, if $N$ is odd,
or over $\{(\rho ,\theta )\mid r_2\leq \rho
\leq r_1,\;(-N+1)\pi\leq \theta \leq \pi (N+1)\}\subset \mathcal{P}$, if $N$ is even.
\item An  $N$-valued graph $u(\rho,\t)$ over the annulus $ A(r_1,r_2)$ is
called {\em righthanded} \, [{\em lefthanded}] if whenever it makes sense,
$u(\rho,\t)<u(\rho,\t +2\pi)$ \, [$u(\rho,\t)>u(\rho,\t +2\pi)$]
\item The set $\{(r_2,\theta, u(r_2,\theta)), \theta\in[-N\pi,N\pi]\}$ when $N$ is odd
(or $\{(r_2,\theta, u(r_2,\theta)), \theta\in[(-N+1)\pi,(N+1)\pi]\}$ when $N$ is even)
is the {\em inner boundary} of the $N$-valued graph.
\een }
\end{definition}

From Theorem 2.23 in~\cite{mt7} one obtains the following,
detailed geometric description
of an $H$-disk with large
norm of the second fundamental form at the origin. The precise meaning
of certain statements below are
made clear in~\cite{mt7} and we refer the reader to that paper for further details.

\begin{theorem}\label{mainextension}
Given $\ve,\tau>0$   and $\ov{\ve}\in (0,\ve/4)$ there
exist  constants $\Omega_\tau:=\Omega(\tau )$,
$\omega_\tau:=\omega(\tau )$ and  $G_\tau:=G(\ve,\tau,\ov\ve ) $
such that if $M$ is an $H$-disk, $H\in (0,\frac 1{2\ve})$,
$\partial M\subset \partial \B(\ve)$, $\vec 0\in M$ and
$|A_M|(\vec 0)>\frac{1}{\eta}G_\tau$,  for $\eta\in(0,1]$, then for any $p\in \ov{\B}(\vec{0},\eta\ov{\ve})$ that is a maximum of the
function
$|A_{M}|(\cdot)(\eta\bar\ve-|\cdot|)$, after translating
$M$ by $-p$, the following geometric description of $M$ holds:
\par

On the scale of the norm of the second
fundamental form  $M$ looks like one or two helicoids nearby  the
origin and, after a rotation that turns these helicoids into
vertical helicoids, $M$  contains a 3-valued graph
$u$ over  $A(\ve\slash\Omega_\tau,\frac{\omega_\tau}{|A_M|(\vec 0)})$
with the norm of its gradient less than $\tau$ and with its inner boundary
in $\B(10\frac{\omega_\tau}{|A_M|(\vec 0)})$.
 \end{theorem}

  Theorem~\ref{mainextension}
 was inspired by the pioneering work of Colding and Minicozzi
 in the minimal case~\cite{cm21,cm22,cm24,cm23};
however in the constant positive mean curvature setting this
description has led to a   different conclusion, that is the
existence of  radius and curvature estimates stated
below.

\begin{theorem}[Extrinsic radius estimates, Theorem~3.4 in~\cite{mt7}] \label{rest}
There exists an
${\mathcal R_0}\geq \pi$ such that for any $H$-disk $\cD$,
$$ {\large{\Large \sup}_{\large p\in \cD }\{d_{\rth}(p,\partial \cD)\} \leq\frac{{\mathcal R_0}}{H}.}$$
\end{theorem}

\vspace{.1cm}

\begin{theorem}[Extrinsic curvature estimates, Theorem~3.5 in \cite{mt7}]
\label{ext:cest}
Given $\delta,\cH>0$, there exists a constant $K_0(\delta,\cH)$ such that
for any $H$-disk $\cD$ with $H\geq \cH$,
 $${{\sup}_{ \{p\in \cD \, \mid \, d_{\rth}(p,\partial
\cD)\geq \delta\}} |A_\cD|\leq  K_0(\delta,\cH)}.$$
\end{theorem}

\vspace{.1cm}

Indeed since the plane and the helicoid are complete
simply-connected minimal surfaces properly
embedded in $\rth$, a radius estimate does not hold in
the minimal case. Moreover rescalings of a helicoid
give a sequence of embedded minimal disks with arbitrarily
large norm of the second fundamental
form at points arbitrarily far
from its boundary; therefore in the minimal setting, the
curvature estimates also do not hold.

The next two results from~\cite{mt9} will be essential tools in
the proof of Theorem~\ref{thm2.1}.

\begin{theorem}[Extrinsic one-sided curvature estimates for $H$-disks] \label{th}
There exist $\ve\in(0,\frac{1}{2})$ and
$C \geq 2 \sqrt{2}$ such that for any $R>0$, the following holds.
Let $\cD$ be an $H$-disk such that $$\cD\cap \B(R)\cap\{x_3=0\}
=\O \quad \mbox{and} \quad \partial \cD\cap \B(R)\cap\{x_3>0\}=\O.$$
Then:
\begin{equation} \label{eq1}
\sup _{x\in \cD\cap \B(\ve R)\cap\{x_3>0\}} |A_{\cD}|(x)\leq \frac{C}{R}.
\end{equation} In particular, if $\cD\cap \B(\ve R)\cap\{x_3>0\}\neq\O$,
then $H\leq \frac{C}{R}$.
\end{theorem}

The next corollary follows immediately from   Theorem~\ref{th} by a simple
rescaling argument. It roughly states that
we can replace the  $(x_1,x_2)$-plane   by any surface that has a fixed
uniform estimate on the norm of its second fundamental form.

\begin{corollary} \label{cest2}
Given an $a\geq 0$,
there exist $\ve\in(0,\frac{1}{2})$ and $C_{a} >0$ such that
for any $R>0$, the following holds.
Let $\Delta$ be a compact immersed surface in $\B(R)$ with
$\partial \Delta \subset \partial \B(R)$, $\vec{0}\in \Delta$
and satisfying $|A_{\Delta}| \leq a/R$. Let $\cD$ be an $H$-disk such that
$$\cD\cap \B(R)\cap\Delta=\O \quad \mbox{and} \quad \partial \cD\cap \B(R)=\O.$$
Then:
\begin{equation} \label{eq1*}
\sup _{x\in \cD\cap \B(\ve R)} |A_{\cD}|(x)\leq \frac{C_{a}}{R}.
\end{equation} In particular, if
$\cD\cap \B(\ve R)\neq \O$, then $H\leq \frac{C_{a}}{R}$.
\end{corollary}

Next, we recall the notion of
flux of an $H$-surface; see for instance~\cite{kks1,ku2,smyt1} for
further discussions of this invariant.

\begin{definition} \label{def:flux} {\rm
Let $\gamma$ be a 1-cycle in an $H$-surface $M$. The  {\em flux} of
$\gamma$ is $\int_{\gamma}(H\gamma+\xi)\times \dot{\gamma}$, where $\xi$
is the unit normal to $M$ along $\gamma$. }
\end{definition}

The flux of a 1-cycle in   an $H$-surface $M$ is a homological invariant and
we say that  $M$ has {\em  zero flux} if the flux of any 1-cycle in $M$ is zero;
in particular, since the first homology group of a disk is zero,   an $H$-disk has  zero flux. 
Finally, we also recall the following definition.

\begin{definition} \label{def:lbsf} {\rm Let $U$ be an open set
in $\rth$. We say that a sequence of surfaces
$\{\Sigma(n)\}_{n\in \N}\subset U$,  has
{\em locally bounded norm of the second
fundamental form in $U$} if for every compact
subset $B$ in $U$, the norms of the second fundamental forms
of the surfaces $\Sigma(n)\cap B$ are
uniformly bounded. } \end{definition}

\section{Proof of  the limit lamination theorem for $H$-disks.} \label{hlam}

\begin{proof}[Proof of Theorem~\ref{thm2.1}] Let $\{M_n\}_{n\in\mathbb N}$ be a
sequence of constant mean curvature disks
that satisfy the hypotheses of Theorem~\ref{thm2.1}.
The extrinsic radius estimates   given  in
Theorem~\ref{rest} imply $\lim_{n\to \infty}H_n=0$.
We will break the proof into the following two cases.\vspace{.2cm}

\noindent {\bf Case A:} $\{M_n\}_{n\in\mathbb N}$
has locally bounded norm of the second fundamental
form in $\rth$.\vspace{.1cm}

\noindent {\bf Case B:} There is a compact set $B$ and the sequence
$\{M_n\}_{n\in\mathbb N}$ has unbounded norm of the second fundamental
form in $B$.\vspace{.1cm}

Some of the arguments in the proofs of Cases A and B listed
above are borrowed from  arguments
 appearing in the proofs of  similar Cases A and B in the proof
of Proposition~3.1  in~\cite{mt9}, where all the details are given.
Consequently, the reader may  wish to  consult the proof of
 Proposition~3.1  in~\cite{mt9} before continuing to read what follows.

We will first prove the theorem under the hypothesis
that Case A holds.  In this case we can follow the
arguments in the proof of Case A of Proposition~3.1 in~\cite{mt9} to conclude that a
subsequence of $\{M_n\}_{n\in\mathbb N}$ converges
with multiplicity one or two to a helicoid $\cH=M_\infty$ containing $\vec{0}$.
 For the sake of completeness we outline the main arguments in the proof of this fact.
 First a standard compactness argument
 using the fact that the surfaces $M_n$ have  uniformly bounded
 norm of the second fundamental form in balls of $\rth$
implies  that after replacing by a subsequence,
the surfaces converge $C^\a$, for any $\a\in (0,1)$,  to a
minimal lamination $\cL$ of $\rth$, and
since $|A_{M_n} |(\vec{0})\geq \de$, then the
leaf $M_\infty$ of $\cL$ passing through the origin $\vec{0}$
has the norm of its second fundamental form bounded
from below by $\de$ at $\vec{0}$.  By Theorem~1.6
in \cite{mr8}, the limit leaves of
 the  minimal lamination  $\cL$ form a closed set ${\bf C}$ of parallel
 planes and every non-flat leaf of $\cL$,
 including $M_\infty$, is properly embedded
 in a component of the collection of open slabs and half-spaces  that
 form the components of $\rth -{\bf C}$.
 For any non-planar leaf $L$ of $\cL$
certain subdomains of  the surfaces $M_n$ converge to it
with multiplicity one or two, as $n$ tends to infinity; otherwise a
higher order convergence
would imply that $L$  is stable and hence planar,
which is a contradiction.
From the multiplicity one or two convergence of the $\{M_n\}_{n\in\mathbb N}$
to the non-flat
leaf $L$, a curve lifting argument implies
that $L$ has genus zero. By the properness of finite genus leaves of a
minimal lamination of $\rth$ (Theorem~7 in~\cite{mpr3}),
$L$ must be properly embedded in $\rth$. Hence,  all of the leaves of
$\cL$ are properly embedded. Since  the leaf
$M_\infty$ is not flat, then the strong halfspace theorem in~\cite{hm10}
implies that $M_\infty$ is the only leaf
in $\cL$.
As the surfaces $M_n$ converge (on compact subsets of $\rth$)
with multiplicity one or two to $M_\infty$,
a standard lifting argument of curves
on $M_\infty$ to the surfaces $M_n$ for $n$
large implies that $M_\infty$ has zero flux,
since the $H_n$-disks $M_n$ have zero flux.
Since properly embedded minimal surfaces
in $\rth$ with zero flux have one end~\cite{cmw1}, $M_\infty$
must have one end; hence  $M_\infty$ is
simply-connected.  By the uniqueness of the
helicoid~\cite{mr13}, $M_\infty$ must be a helicoid (see also~\cite{bb1}
for a proof).
Therefore, we conclude that a
subsequence of the original sequence $\{M_n\}_{n\in\mathbb N}$ converges
with multiplicity one or two to a helicoid $\cH=M_\infty$ containing $\vec{0}$.

 We now check that for
$n$ large every component of $M_n\cap \B(1)$
is an open  subdisk of the disk $M_n$ whose closure is a compact disk
with piecewise smooth boundary. Otherwise, after replacing by a subsequence,
for $n$ sufficiently large there exists a component
$\Delta_n$ of $M_n\cap \B(1)$ whose closure
$\ov{\Delta}_n\subset \Int(M_n)$
is not  a compact disk with piecewise smooth boundary. In the case where $M_n$  intersects
$\partial \B(1)$ transversely,  $M_n\cap \oB(1)$ is a smooth compact surface
that is the closure of  $M_n\cap \B(1)$ and the following arguments
can be simplified; therefore, on a first reading of the next paragraph
the reader might want to consider this special generic case first.

Since $M_n\cap \oB(1)$ is a two-dimensional semi-analytic
set in $\rth$ and $M_n\cap \partial \B(1)$ is an analytic subset of the sphere
$\partial \B(1)$, then, by \cite{lo1}, $M_n\cap \oB(1)$
admits a triangulation  by analytic  simplices, and the interiors of
the 2-dimensional simplices are contained in
$\B(1)$ because otherwise by analyticity then
$M_n \subset\partial \B(1)$ which is false. Since the inclusion map
of $M_n$ is an injective immersion, then it follows that  $\ov{\Delta}_n$
is a semi-analytic subset of $M_n$ that can be triangulated with
a finite number of closed 2-dimensional analytic simplices whose
interiors are contained in ${\Delta}_n\subset \B(1)$  and
the component $\ov{\Delta}_n\cap \partial \oB(1)$ is  a connected 1-dimensional
analytic subset of $M_n$, where we identify $M_n$ with its image in $\rth$;
note that $\ov{\Delta}_n\cap \partial \oB(1)$ does not contain any isolated points
by the mean curvature comparison principle.  By the
elementary topology of the disk $M_n$ and using arguments as in~\cite{my1}, one can check that
$\ov{\Delta}_n$ fails to be a
disk with piecewise smooth analytic boundary
if and only if there exists a
simple closed piecewise analytic  curve $\G(n)  $ contained in the 1-dimensional simplicial sub-complex
of $ \ov{\Delta}_n \cap \partial \B(1)$ such that  $\G(n) $
does not bound a  disk in $\ov{\Delta}_n$. In the case that $M_n$
is transverse to $\partial \B(1)$, then $\G(n)$ can be chosen to be the boundary curve
of  a   component of $M_n\cap (\rth -\B(1)) $ that has its entire boundary in $\partial \B(1)$.

Arguing by contradiction, suppose   $\ov{\Delta}_n$ is not a compact disk.
Let $D_n$ denote the compact subdisk of $M_n$ with
boundary $\G(n)\subset M_n\cap \partial \B(1)$ and notice
that $D_n\not \subset\oB(1)$.  Hence, there is a point $p_n\in D_n$ that
has maximal distance $R_n>1$ from the origin. Since the boundary of $D_n$ lies
in $\partial B(1)$ and $D_n$ lies in $\rth - \oB(1)$ near $\partial D_n$,
then $p_n$ is an interior point of $D_n$ not contained in $\B(1)$
and $D_n$ lies inside the closed ball $\ov{\B}(R_n)$ and intersects
$\partial \B(R_n)$ at the point $p_n$.  By the mean curvature
comparison principle applied at the point $p_n$,
the constant mean curvature of $M_n$ is at least $1/R_n$
and so, since the constant
mean curvature values of the surfaces $M_n$ are tending to zero as $n$ goes to infinity,
the interior points $p_n\in D_n$
are diverging to infinity in $\rth$ as $n$ goes
to infinity. Let $p$ be a point on the
axis of $\cH$ that is closest to the origin. Let $r_0=|p|$ be the distance
from $p$ to the origin and let $E$ be the smooth compact disk
$\cH\cap \ov{\B}(p,r_0+2)$. Then because of the
multiplicity one or two convergence of the surfaces $M_n$ to $\cH$, for
$n$ large,  $M_n$ contains a unique smooth disk
subdomain $M_E(n)$ that contains the origin
and that is a small normal graph
over $E$, and in the case the multiplicity of convergence is two,
 $M_n$ contains another smooth disk
subdomain $M_E'(n)$ that does not contain the origin
and that is a small normal graph
over $E$; furthermore, for $n$ sufficiently large,
 every point of $M_n\cap \B(1)$ is contained in
$M_E(n)\cup M_E'(n)$.
In particular, for $n$ large,
$\Delta_n\subset (M_E(n)\cup M_E'(n))\subset \B(p,r_0+3)$,
and so $\G(n)\subset M_E(n)$ or  $\G(n)\subset M_E'(n)$.  If $\G(n)\subset M_E(n)$
then it bounds in $M_E(n)$ a disk $D'_n$,
which   must be the disk $D_n$; if $\G(n)\subset M_E'(n)$
then it bounds in $M_E'(n)$ a disk $D'_n$,
which again must be the disk $D_n$. This is a contradiction since
$ D_n \not\subset \B(p,r_0+3)$ for $n$ large, but
$D_n'\subset (M_E(n)\cup M_E'(n))\subset \B(p,r_0+3)$.
This proves that
every component $\Delta_n$ of $M_n\cap \B(1)$
is an open  subdisk of the disk $M_n$ and $\ov{\Delta}_n$ is a compact
disk with piecewise smooth boundary in $\partial \B(1)$.

To complete the proof of  Theorem~\ref{thm2.1} when Case A holds, it
remains to prove the upper bound of 10 on the intrinsic distances on $M_n$
between any two points on a fixed
disk $\Delta_n$ component of $M_n\cap \B(1)$ (and then by the continuity
the intrinsic distance function on $M_n$, the same property holds
for points in $\ov{\Delta}_n$).  This property
follows easily because it holds for the helicoid $\cH$; this completes the proof of
the theorem when Case A holds.

Assume now that Case B holds, i.e., the sequence $\{M_n\}_{n\in\mathbb N}$
does {\em not} have
bounded norm of the second fundamental form in some compact subset
$B\subset \rth$. We will prove that item~(B) of Theorem~\ref{thm2.1} holds in
this case. After
replacing $\{M_n\}_{n\in\mathbb N}$ by a subsequence, there is a point $x\in \rth$
closest to $\vec{0}$ such that there exist
  points $x_n \in M_n$  with $|A_{M_n}|(x_n)>n$ and $\lim_{n\to\infty} x_n=x $.
After another replacement by a subsequence, a straightforward
application of
Theorem~\ref{mainextension} produces a sequence of
3-valued graphs
$G_n\subset M_n$ with inner boundaries near $x_n$ that converge  to
a plane $P_x$ passing through $x$  (see the proof of
Claim~3.5 in~\cite{mt9} for additional details on the construction of the
graphs $G_n\subset M_n$
that converge to $P_x$).
Next choose a rotation
$T\colon \rth \to \rth$ such that $T(x)=\vec{0}$ and such that the
plane $T(P_x)$ is the $(x_1,x_2)$-plane, and replace the surfaces
$M_n$ by the rotated surfaces $T(M_n)$.
With an abuse of notation we will
let $M_n$ denote the rotated surfaces  $T(M_n)$ and note that
these new surfaces may possibly not pass through the
origin. Note that in Claim~\ref{vertline} we will prove that $x=\vec 0$.

After this replacement of $M_n$ by $T(M_n)$, $P_{\vec{0}}$ is the $(x_1,x_2)$-plane.
By the one-sided curvature estimates given in Theorem~\ref{th}
and in Corollary~\ref{cest2}, outside some closed solid vertical cone
$\mathcal C_{\vec 0}$ based at $\vec{0}$
of a certain fixed aperture determined only by
the curvature estimates given in Theorem~\ref{th} and Corollary~\ref{cest2},
the sequence of surfaces $\{M_n\}_{n\in\mathbb N}$
has locally bounded norm of the second fundamental
form.   After replacing $\{M_n\}_{n\in\mathbb N}$ by a subsequence, we may
assume that the sequence
$\{M_n\}_{n\in\mathbb N}$ converges to a minimal lamination  $\cL$ of $\rth- \cS$,
where $\cS$ is a smallest closed
set such that the sequence $M_n$
has locally bounded norm of the second fundamental form in
$\rth- \cS$, and no subsequence of these surfaces has a smaller
singular set of convergence; see for example, the proof of Claim~3.4 in~\cite{mt9}
for the diagonal type  argument leading to
this statement.  By the previous discussion in this
paragraph, through each point $y\in \cS$ there
passes a horizontal plane $P_y$ and $\cS\cap [\rth-\cC_y]=\O$,
where  $\cC_y$ is the vertical cone previously
described  but based at $y$ instead of at $\vec{0}$. As already observed, the one-sided curvature estimates
given in Corollary~\ref{cest2} imply that $\cS\cap P_y=\{y\}$.

Fix $p\in \cS$. We claim that there exist points $p_n\in M_n$,
with vertical tangent planes and such that $\lim_{n\to \infty}p_n= p$.
Otherwise in some small neighborhood of $p$, $M_n$ admits a
non-zero Jacobi function induced by the parallel Killing field $E_3=(0,0,1)$,
which would give curvature estimates of
$M_n$ near $p$ which is a contradiction.
Also this
geometric description gives that if $p_n\in M_n$,
$p_n$ is a sequence of points with vertical tangent planes
that is converging to a point $q\in\rth$,
then there exists a sequence of points
$q_n \in M_n$ with $\lim_{n\to\infty}d_{M_n}(p_n,q_n)= 0$ and with
$\lim_{n\to\infty}|A_{M_n}|(q_n) = \infty$.
Otherwise, after replacing by a subsequence,
a geodesic ball of fixed radius centered at $p_n$
would satisfy curvature estimates and
thus be a graph over a vertical plane.
By Corollary~\ref{cest2}, the existence of this graph would lead to a
contradiction because such a graph cannot
be contained in either $\cS$ or in a horizontal plane. In particular,
this implies that $q\in \cS$. In other words, if $\mathcal Z$ denotes
the set of points in $\rth$ that are limits of points $p_n\in M_n$
with vertical tangent planes, then this discussion implies that $\mathcal Z =\cS$.

The
following claim is analogous to the statement
of Corollary~3.8 in~\cite{mt9} and indeed
the proof is the same by  using the above observation
that intrinsically close to a point where
the tangent plane is vertical, there exist points where the norm of the second
fundamental form is arbitrarily large. We refer the reader to the proof of
Corollary~3.8 in~\cite{mt9} for details in the proof of the next result.

\begin{claim}\label{convhel}
Given $\Lambda, R,\ve, k>0$, there exists $N=N(\Lambda,R,\ve, k)\in \N$
such that the following holds. Let $\cH$ denote a vertical helicoid
containing the origin
with maximal absolute Gaussian
curvature $\frac12$ at the origin and let  $p_n\in M_n\cap \B(R)$
be a sequence of points where the tangent planes  are vertical.
For any $n>N $,  $|A_{M_n}|(p_n)>k$ and there exists a
rotation $T_n$ about the $x_3$-axis such that  the intersection set
$(|A_{M_n}|(p_n)[M_n-p_n])\cap \B(\Lambda)$ consists of one or two
connected components and each component
consists of a normal graph $u$ over its projection to
$ T_n(\mathcal H)$ or $-T_n(\mathcal H)$ and $\|u\|_{C^2}\leq \ve$.
\end{claim}

Since $\cS$ is a closed set in $\cC_{\vec 0}$ and the mapping
$x_3|_{\cC_{\vec 0}} \to \R$ is  proper,
it follows that $x_3(\cS)$ is a closed subset of $\R$ and
so $\R-x_3(\cS)$ is a possibly empty collection
of pairwise disjoint open intervals together with perhaps one or
two half-lines.
\begin{claim}\label{vertline} The intersection $\{x_3=t\}\cap \cS$ is a single point for every $t\in \R$. In fact, the set $\mathcal S$  is a vertical line passing through the
origin and $x=\vec{0}$.
\end{claim}
\begin{proof}
We first prove that the intersection $\{x_3=t\}\cap \cS$ is a single point for every $t\in \R$.
Arguing by contradiction, suppose that for some $t_0 \in \R$,
$\{x_3=t_0\}\cap \mathcal S= \O$. The point $t_0 $
is contained in a  component, $I=(a,b) \subset[\R-{x_3(\cS)}]$,
where perhaps $a$ or $b$ is equal to $\pm\infty$ but
not both. Without loss of generality, we may assume that $a\neq-\infty $
and $b\in ( a,\infty]$ and define $W=x_3^{-1}((a,b))$;
the set $W$ is either an open slab or an open
half-space.
Let $\Delta$ be the  set of  points
in $\partial W\cap \mathcal S$. By the previous discussion, $\Delta$
contains exactly one point of $\cS$ in each  component
of $\partial W$.
Since the sequence $\{M_n\cap W\}_n $ has uniformly locally
bounded norm of the second
fundamental form in $W$ and $\lim_{n\to \infty} H_n = 0$, the sequence converges to the minimal lamination
$\mathcal L_W=\cL\cap W$ of $W$. Again
the one-sided curvature estimates  in
Theorem~\ref{th} and Corollary~\ref{cest2}  give that
$[\cL_W\cup \partial W]-\Delta$ is a minimal
lamination of $\rth-\Delta$. Note that
$\cL_ W\neq \O$, otherwise after  placing a horizontal plane in $W$ disjoint
from $\cL$ and arbitrarily close to $\Delta$, one could apply the one-sided curvature
estimate in Theorem~\ref{th} to show
that $\Delta=\O$.
As already observed, the closure $\ov{\cL}_W$
of $\cL_W$ in $\rth -\Delta$ has the structure of a minimal lamination
of $\rth-\Delta$, and it is obtained by adding the  leaves $\partial W -\Delta$
to the lamination $\cL_ W$.

Following the arguments at the end of the proof of Proposition~3.1
in~\cite{mt9}, it is straightforward to check that the leaves of
the lamination  $\ov{\cL}_W$ have genus zero. Briefly, to see this property holds
 one first shows that
none of the leaves in this lamination contained  in $W$ are stable,
which by the Stable Limit Leaf Theorem
in~\cite{mpr19,mpr18} implies every leaf $L$ of
$\cL_W$ is  nonflat and proper in $W$ and the convergence of the surfaces
$M_n$ to $L$  has multiplicity
one or two. Finally, a curve lifting argument shows that $L$ has genus zero.

Let $L$ be a proper nonflat genus-zero leaf in $\cL_W$.
Since we can view  $\ov{\cL}_W$ to be
a minimal lamination of $\rth$
with a countable number of singularities (the
singularities being in the finite set $\Delta$),
then item~6 in Theorem~1.8  of~\cite{mpr11} implies
that the closure $\ov{L}$ of  $L$ in $\rth$
is a properly embedded minimal surface in $\rth$
which lies on one side of one of the boundary planes of $W$.
(In fact, to apply item~6 in Theorem~1.8
of~\cite{mpr11}  it suffices that the genus of $L$ is finite.)
But $L$ is contained in a half-space, which  contradicts
the half-space theorem in~\cite{hm10}.
  This contradiction  implies that for any $t\in\mathbb R$, $\{x_3=t\}\cap \mathcal S\neq\O$,
and, as previously observed, then $\{x_3=t\}\cap \mathcal S$ is a single point.

Recall again that by the one-sided curvature estimates given in
Theorem~1.1, for any $y\in\cS$,    $\cS\cap [\rth-\cC_y]=\O$
where $\cC_y$ is the vertical cone based at $y$ previously described.
Therefore $\cS$ is a
Lipschitz curve parameterized by its $x_3$-coordinate and
 the sequence
$\{M_n\}_{n\in\mathbb N}$ converges to a minimal foliation
$\cL$ of $\rth- \cS$ by punctured horizontal planes.
Since before the replacement of $M_n$ by the rotated surfaces $T(M_n)$
the norms of the second fundamental forms of the surfaces
were bounded away from zero at $\vec{0}$, the nature of the $C^\a$ convergence
of the surfaces $T(M_n)$ to the flat planes in $\cL$ outside of $\cS$ implies that
$T(\vec{0})\in \cS$; hence, the originally chosen  singular
point $x$ of convergence ``closest" to $\vec{0}$ must be   $\vec{0}$ itself
and therefore $T(\vec{0})=\vec{0}$.

We next check that the curve $\mathcal S$  is a vertical line passing through the
origin.
Let
\[
\Gamma^R_n=N_n^{-1}(\{x_1^2+x_2^2=1\})\cap \B(R)
\]
where $N_n$ denotes the Gauss map of $M_n$.
By the discussion before Claim~\ref{convhel}, the set
$\Gamma^R_n$ is nonempty and converges to $\cS\cap \B(R)$ as
$n$ goes to infinity. Therefore, it suffices to show that
$\Gamma^R_n$ converges $C^1$, possibly with multiplicity, to the line segment
$\{(0,0,t)\mid t\in(-R,R)\}$. This follows by
Claim~\ref{convhel} because for $n$ large, $\Gamma^R_n$ is
an analytic curve with tangent lines converging uniformly to
the $x_3$-axis in balls centered at the origin of any fixed radius.
\end{proof}

The previous claims complete the proof of part (a) of item~B
in the theorem.  It remains to prove the last
statement in the theorem.

\begin{claim}\label{twocomponents}
Let $\Gamma_n=N_n^{-1}(\{x_1^2+x_2^2=1\})\cap \B(3)$.
Then $\Gamma_n$ consists of one or two analytic curves for $n$ large.
\end{claim}
\begin{proof}
By Claim~\ref{vertline} and its proof, the set $\Gamma_n$
is a possibly disconnected analytic curve
that converges $C^1$ to the line segment
obtained by intersecting the  $x_3$-axis with $\B(3)$. Hence,
it suffices to prove that $\Gamma_n\cap\{x_3=0\}$
consists of one or two points for $n$ large.
Suppose that, after replacing by a subsequence,
$p_1(n), p_2(n), p_3(n)\in \Gamma_n\cap\{x_3=0\}$
are three distinct points converging to $\vec{0}$
and let $\Gamma_n(i)$ be the connected component
of $\Gamma_n$ containing $p_i(n)$, $i=1,2,3$.
Without loss of generality, after applying small
horizontal translations and rotations around the $x_3$-axis, we may assume that
$p_1(n)=\vec 0$, $p_3(n)=(t_n,0,0)$ and $|p_2(n)|\leq t_n$ satisfying
$\lim_{n\to \infty} t_n=0$.

Consider the subsequence of rescaled disks $\S_n=\frac{1}{t_n} M_n$
that contain the points $\vec{0}$ and $(1,0,0)$
and let
$q(n)=\frac{1}{t_n}{p_2(n)}$. Note that
$\vec 0,(1,0,0),q(n) \in \S_n\cap \{ x_3=0\}$, the tangent planes to $\S_n$
at $\vec 0$, $q_n$ and $(1,0,0)$ are vertical and $|q(n)|\leq 1$.
Claim~\ref{convhel} implies that a subsequence of the homothetically scaled surfaces $ |A_{M_n}|(\vec 0)M_n$
converges to a vertical helicoid $\cH$ containing the $x_3$-axis. We are going to use what
we have proven so far to analyze three exhaustive cases
that would occur after replacing by a subsequence.
\begin{enumerate}
\item $\lim_{n\to\infty}|A_{\S_n}|(\vec 0)= 0$;
\item $\lim_{n\to\infty}|A_{\S_n}|(\vec 0)= C\in (0,\infty)$;
\item $\lim_{n\to\infty}|A_{\S_n}|(\vec 0)= \infty$.
\end{enumerate}

If $\lim_{n\to\infty}|A_{\S_n}|(\vec 0)=\lim_{n\to\infty}t_n |A_{M_n}|(\vec 0)= 0$,
then $ |A_{M_n}|(\vec 0)M_n$ converges
to $\cH$ with multiplicity greater than two. This is
because $|p_2(n)|\leq t_n$ implies
there are at least three points in $ |A_{M_n}|(\vec 0)M_n\cap \{x_3=0\}$
converging to the origin and having
vertical tangent planes. However, a multiplicity greater than
two convergence implies stability.
Since the helicoid is not stable, we have obtained a
contradiction when case~1 holds.

Next, assume that
$\lim_{n\to\infty}|A_{\S_n}|(\vec 0)
=\lim_{n\to\infty}t_n |A_{M_n}|(\vec 0)= C\in (0,\infty)$.
Again,
by Claim~\ref{convhel} we have that a subsequence of
$ |A_{M_n}|(\vec 0)M_n$ converges to a vertical helicoid $\cH$
containing the origin and
\[
\S_n=\frac{1}{t_n} M_n=\frac{1}{t_n|A_{M_n}|(\vec 0)} |A_{M_n}|(\vec 0) M_n
\] must converge to the vertical helicoid $\frac 1C \cH$.
However, this leads to a contradiction
because $(1,0,0)$ is a point in $\S_n$ with
vertical tangent plane. This proves the claim when case~2 holds.

Finally, assume that $\lim_{n\to\infty}|A_{\S_n}|(\vec 0)=\lim_{n\to\infty}t_n |A_{M_n}|(\vec 0)=\infty$.
By our previous discussion, there exists a foliation $\cF$ of $\rth$ by planes and a line $\cS$
orthogonal to the planes in $\cF$ such that $\vec 0\in \cS$ and
after replacing by a  subsequence,  the new sequence of surfaces $\S_n$ converges
$C^\a$, for $\a\in (0,1)$, to the foliation
$\cF-\cS$ in $\rth-\cS$.
We claim
that $\cF$ must be a foliation by vertical planes. If $(1,0,0)\notin \cS$,
then the claim is true because the tangent plane to $\S_n$ at
$(1,0,0)$ is vertical. If $(1,0,0)\in \cS$ then, since
$\vec 0\in \cS$ and $\cS$ is a straight line orthogonal
to the planes in the foliation, $\cS$ is the $x_1$-axis
and $\cF$ is a foliation by vertical planes.

 Recall that  as $n$ goes to infinity,
the boundary curves of the disks $M_n$
are converging to infinity and the constant values of
their mean curvatures are going to zero.
For $n$ sufficiently large, let $\widetilde M_n\subset M_n$ be
the $H_n$-subdisk containing the origin with
$\partial \widetilde M_n\subset \partial \B(1)$. Without loss of
generality we can assume $H_n \in (0, \frac 1{2})$.
Let $G_1, \omega_1, \Omega_1$ be the values of the functions
$G_\tau, \omega_\tau, \Omega_\tau$ given by
Theorem~\ref{mainextension}, after fixing $\ve=\tau=1$, $\ov\ve=\frac 12$.
Note that since
$
\lim_{n\to\infty}t_n |A_{\widetilde M_n}|(\vec 0)= \infty,
$ for $n$ sufficiently large,
$|A_{\widetilde M_n}|(\vec 0)\geq \frac 1{t_n} G_1$. Thus,
when $n$ is sufficiently large, Theorem~\ref{mainextension} and the fact that $M_n$ converges to a foliation of $\rth$ minus the $x_3$-axis by horizontal planes,
imply that there exist points $p_n \in \wt{M}_n\cap\B(\frac{t_n}2)$ such
that the following holds: $\widetilde M_n-p_n$,
and thus $M_n-p_n$,  contains a 3-valued graph
$u_n$ over  $A(1\slash\Omega_1,\frac{\omega_1}{|A_{M_n}|(\vec 0)})$
with the norm of its gradient less than $1$ and with its inner
boundary in $\B(10\frac{\omega_1}{|A_{M_n}|(\vec 0)})$.

Note that since $p_n\in\B(\frac{t_n}{2})$, a subsequence of the
translated and scaled surfaces $\Sigma_n'=\frac 1{t_n} [M_n-p_n]$
must converge to the translated foliation $\cF'=\cF-\vec v$ away
from the translated singular set $\cS'=\cS-\vec v$,
where $\vec v=\lim_{n\to\infty}\frac{p_n}{t_n}$, $|\vec v|\leq \frac 12$.
In particular, the leaves of $\cF'$ are vertical planes and
the singular set $\cS'$ is a horizontal line perpendicular to
the leaves of $\cF'$. However, for $n$ large, $\Sigma_n'$
contains a 3-valued graph $u_n'=\frac 1{t_n}u_n$ over
$A(1\slash t_n\Omega_1,\frac{\omega_1}{t_n|A_{M_n}|(\vec 0)})$
with the norm of its gradient less than $1$ and with its inner
boundary in $\B(10\frac{\omega_1}{t_n|A_{M_n}|(\vec 0)})$.
Since $\lim_{n\to\infty}t_n=0$ and $\lim_{n\to\infty}t_n |A_{\widetilde M_n}|(\vec 0)= \infty$,
this contradicts the fact that the leaves of $\cF'$ are punctured
vertical planes. This contradiction finishes the proof of the claim.
 \end{proof}

The next claim follows from arguments similar to those
in the proof of the last claim and by
Claim~\ref{convhel}, after taking $\L=4$ and $\ve$ sufficiently small.

\begin{claim}\label{two-components}
After replacing by a subsequence, assume that the number of components
in $\G_n$ does not change.
For any fixed
$\theta_0\in [0,1)$, there exists an $N(\theta_0)\in\mathbb N$
such that, after replacing by a subsequence,  the following holds for $n\geq N(\theta_0)$.
For any $\t\in [-\t_0, \t_0]$, let
\[
\gamma_n^\theta := M_n\cap N_n^{-1}(\{x_1^2+x_2^2+x_3^2=1\}\cap \{ x_3=\theta\})\cap \B(3),
\]
then $\gamma_n^\t$ consists of one or two disjoint analytic curves,
the number of these curves does not depend upon the choice of $\t$
(and so this number is the same as the number of components in $\G_n$),
the curves $\gamma_n^\theta$ can be parameterized
smoothly by their $x_3$-coordinates and as $n$ goes to infinity
the  curves $\gamma_n^\theta $ and the sets
$\cup_{\tau\in [-\t_0, \t_0]}\gamma_n^\tau $ converge to the
interval $\{(0,0,t) \mid t\in (-3,3)\}$.
\end{claim}

We next explain
how to construct the set $\cC_n$ consisting of one
or two disk components described in
part (b) of item~B of Theorem~\ref{thm2.1}.
This explanation will complete the proof of the theorem.

Fix a number $\be\in (0,1/3)$.
Let us first suppose that, after choosing a subsequence,    for all $n\in \N$,
$\Gamma_n$ is a single curve, which  we denote by $\g(n)$. The curve $\g(n)$ can be
parameterized by its $x_3$-coordinates that lie in
some interval $(-3+\ve^1_n, 3-\ve^2_n)$,
where, as $n$ goes to infinity,  the
sequences of numbers $\ve^1_n,\ve^2_n\in [0,\be)$
converge to zero and $\g(n)$ is $\be$-close to
the $x_3$-axis.
Then, by Claim~\ref{convhel} and by Claim~\ref{two-components}
with $\L$ chosen  sufficiently large
and $\ve$ chosen sufficiently small,
there exists a $N(\be)\in \N$ such that for $n\geq N(\be)$ the following holds:\\
For each $t\in (-1-\be,1+\be)$, the vertical  tangent plane $T_{\g(n)(t)} M_n$ intersects
$(M_n-\g(n)(t))\cap \oB(1+2\beta)$ transversely, where $M_n-\g(n)(t)$ is the set difference,
in a set containing two components that are arcs such that the closure
of their union  is an analytic
curve $\a_t^n$ with $\g(n)(t)\in\a_t^n$ and  $\a_t^n$   is $\beta$-close to a line
segment contained in
$\tau_t^n= T_{\g(n)(t)} M_n\cap \{x_3=t\}\cap \oB(1+3\beta)$ in the $C^1$-norm.

For $n\geq N(\be)$, define the following piecewise smooth disks:

$$\cC_n^\be=\bigcup_{t\in (-1-\be,1+\be)}\a^n_t, \quad F_n^\be=\bigcup_{t\in (-1-2\be,1+2\be)}\tau^n_t,
$$
\[
 \wh{F}_n^\be=\bigcup_{t\in (-1-2\be,1+2\be)}[\tau_t^n -(x_1(\g(n)(t)),x_2(\g(n)(t)),0)].
\]

Let $I_n\colon F_n^\be \to \wh{F}_n^\be$ be the associated diffeomorphism  induced by
translating the horizontal line segment  $\tau_t^n$ to the horizontal line segment
$\tau_t^n -(x_1(\g(n))(t),x_2(\g(n)(t)),0)$ at the same height. Note that $I_n(\g(n)(t))=(0,0,t)$ for
all $t\in (-1-2\beta, 1+2\beta)$. Claim~\ref{convhel} and the proof of Claim~\ref{vertline}
imply that
$N(\be)$ can be chosen large enough so
that the following condition  also holds:\\
 For   all  unit length
tangent vectors $v$ to $\tau_t^n\subset F_n^\be$ or to the arc $\g(n)\cap F_n^\be$,

\begin{equation} \label{c1}
\frac{1}{|(I_n)_{*}(v)|} \in [1,1+\be],
\end{equation}
where $(I_n)_{*}(v)$ denotes the tangent vector of the related image smooth arcs.

Claim~\ref{convhel}  implies that  $N(\be)$ can also be chosen large enough so
that the following condition also holds:\\
 For $n\geq N(\be)$, there is
a unique injective map $\Pi_n\colon \cC^\be_n\to   F_n^\be$
defined at a point $(x, y, z)\in \a_t^n$ by $$\Pi_n((x, y, z))=(x, y, t).$$
The map
$\Pi_n$   is smooth at points in $\cC_n^\be-\g(n)$,
 $\Pi_n$ restricted to $\g(n)\cap \cC_n^\be$ is the identity function and
for all  unit tangent
vectors $v$ to the pair of arcs $\a_t^n -\{\g(n)(t)\}$ or  to  $\g(n)\cap \cC_n^\be$,

\begin{equation} \label{c2}
\frac{1}{|(\Pi_n)_{*}(v)|}\in  [1,1+\be],
\end{equation}
where $(\Pi_n)_{*}(v)$ denotes the tangent vector of the related image smooth arcs.

We next define a distance function $D$ on $\rth$ that
makes it into a metric space  and that is useful
for estimating distances $D$ between pairs of points in the disks $\cC_n^\be$, see Figure~\ref{Dball}.

\begin{definition} \label{d} {\rm
Let $x=(x_1,x_2,x_3), y=(y_1,y_2,y_3)$ be points in $\rth$.
\ben \item  If $x_3=y_3$,
then $$D(x,y)=\sqrt{(y_1-x_1)^2+(y_2-x_2)^2}.$$
\item Otherwise,
$$D(x,y)=\sqrt{x_1^2+x_2^2} +|x_3-y_3|+\sqrt{y_1^2+y_2^2}.$$
\een
Given a point $p\in \rth$, we let $B^D(p,R)$ denote the open ball center at $p$ of radius $R$.}
\end{definition}

\begin{figure}[h]
\begin{center}
\includegraphics[width=11.5cm]{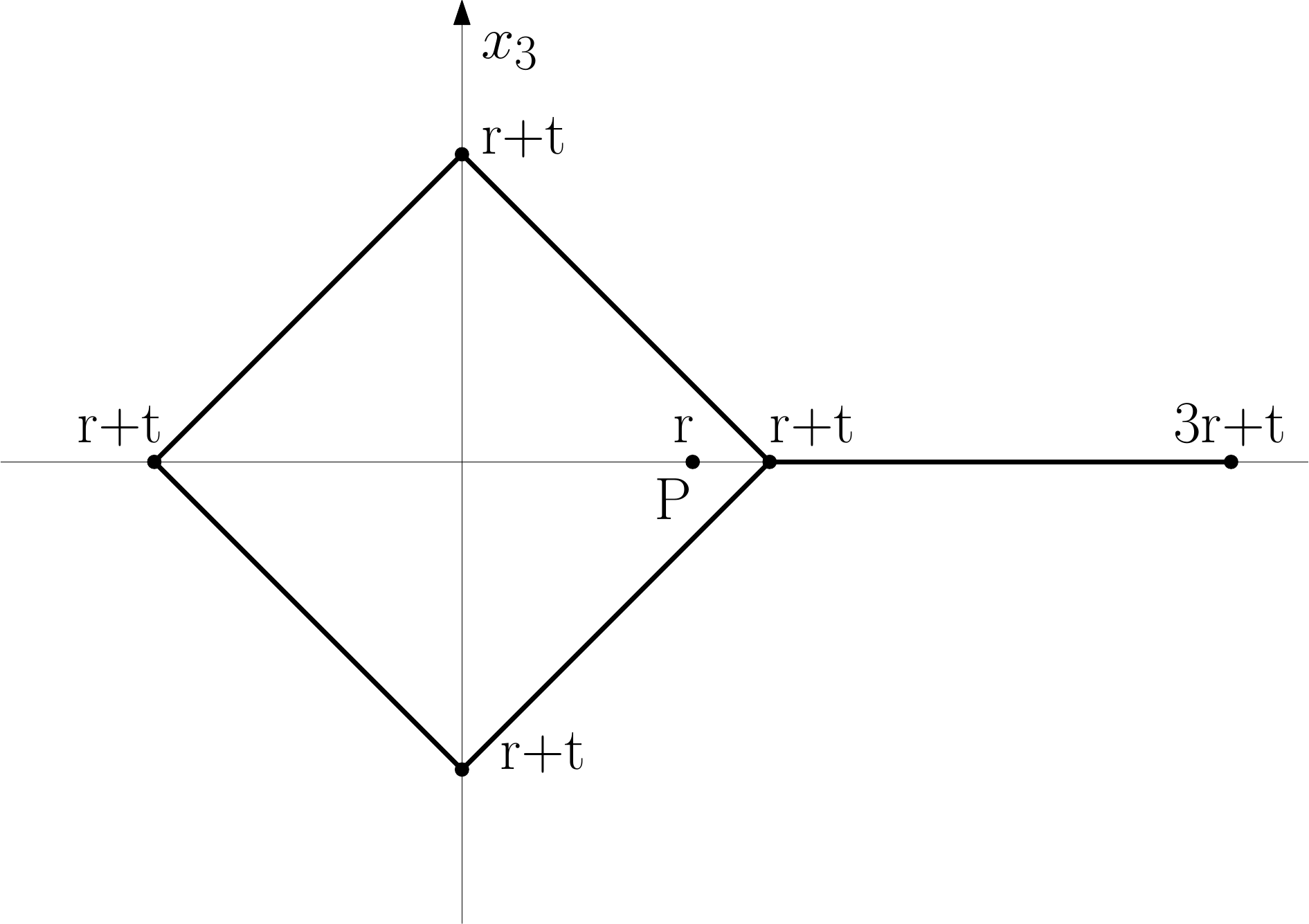}
\caption{The point $P$ is at distance $r$ from the $x_3$-axis. Given $t>0$,
the thick line represents a vertical section of the topological boundary
of $B^D(P,2r+t)$ containing $P$ and the $x_3$-axis.
For $s\in (0,r+t]$, the intersection $B^D(P,2r+t)\cap \{x_3=x_3(P)\pm s\} $
consists of an open disk of radius $r+t-s$ centered at $(0,0,x_3(P)\pm s)$.
 The intersection $B^D(P,2r+t)\cap \{x_3=x_3(P)\}$ consists of the open disk
 centered at $P$ of radius $2r+t$.
It is easy to see that if $Q\notin  B^D(P,2r+t)$ then $|Q-P|\geq \frac{\sqrt2}2 t$.
}
\label{Dball}
\end{center}
\end{figure}

Geometrically speaking, this distance function is defined at a
pair of points $p_1,p_2\in \rth$ that  lie on a horizontal line that intersects
the $x_3$-axis
to be  equal to the length of the line segment between them. Otherwise
it is equal to the length of the piecewise polygonal arc formed
by the horizontal line segment joining $p_1$ to the $x_3$-axis, the horizontal
line segment joining  $p_2$ to the $x_3$-axis and the line
segment in the $x_3$-axis that joins the end points of
these two horizontal line segments.
Note that $D$ induces a
metric space structure with distance function
$D_{\wh{F}_n^\be}$ on  $ \wh{F}_n^\be$ that is greater than or
equal to the Riemannian distance function on the surface $\wh{F}_n^\be$.

Let $x,y\in \cC_n^\be$ be  distinct points. If $x,y$ lie on the same
$\a_t^n$ arc in $\cC_n^\be$, then let $c_{x,y}^{\be,n}$ be
the subarc with end points $x,y$.
Otherwise, consider the embedded piecewise smooth
path $c_{x,y}^{\be,n}$ in  $\cC_n^\be$ formed by the unique
two subarcs in the $\a_t^n$-type curves
joining, respectively, the points $x,y$
to points in $\g(n)$ together
with the subarc in $\g(n)$ that joins the respective end points of these arcs in $\g(n)$;
the reader should note the similarity of
the construction of this piecewise smooth path  joining $x$ and $y$
with the construction of the piecewise polygonal arc in the previous
paragraph that joins  points $p_1$, $p_2$ in $\rth$ and where now $\g(n)$
plays the role of the $x_3$-axis.
Let
\begin{equation}\label{mapg}
G_n^\be\colon \cC_n^\be\to \wh F_n^\be, \quad \mbox{\rm where} \quad G_n^\be:=I_n^\be \circ\Pi_n^\be,
\end{equation}
and given $x,y\in \cC_n^\be$, let
$C_{x,y}^{\be,n}$ denote the embedded piecewise smooth path
$G_n^\be(c_{x,y}^{\be,n})$ that has end points $G_n^\be(x),G_n^\be(y)\in \B(1+4\be)$.
By equations \eqref{c1} and \eqref{c2},

\begin{equation} \label{c3}
\mbox{\rm Length}(c_{x,y}^{\be,n})\in
[\mbox{\rm Length}(C_{x,y}^{\be,n}),(1+\be)^2\mbox{\rm Length}(C_{x,y}^{\be,n})].
\end{equation}

Note that $\mbox{\rm Length}(C_{x,y}^{\be,n})=D(G_n^\be(x),G_n^\be(y))$.
  An elementary calculation shows that
for any points $X,Y\in \oB(R)$, $D(X,Y)\leq 2\sqrt{2}R$. Hence, by
equation \eqref{c3}, we have the estimate

\begin{equation} \label{c4}
\mbox{\rm Length}(c_{x,y}^{\be,n})\leq (1+\be)^2\mbox{\rm Length}(C_{x,y}^{\be,n})
\leq(1+\be)^2 2(1+4\be)\sqrt{2}.
\end{equation}
Since as $n$ goes to infinity,  $\beta$ can be chosen arbitrarily
small, then for $n$ sufficiently large,
we can assume that $(1+\be)^2 2(1+4\be)\sqrt{2}< 3$.
With this choice of $\beta$,
$$\mbox{\rm Length}(c_{x,y}^{\be,n})< 3$$
and thus, the intrinsic diameter of $\cC_n^\be$ is less than 3 for $n$  sufficiently large.
Henceforth, we will  assume that
$\be$ is  chosen sufficiently small and $n$ is
 chosen sufficiently large  so that this inequality holds.

In the case  being considered where $\G_n$ is a single curve, after replacing
by a sequence, define $\cC_n:=\cC_n^\be$ and note that  since $\vec{0}\in C_n $,
$\cC_n \cap \B(1/n)\neq\mbox{\rm \O}$.  By construction,
$\partial \cC_n  \subset [\oB(1+2\be)-\oB(1+\frac23 \be)]\subset [\rth -\oB(1+\frac23 \be)]$.

It remains to prove that
for $n$ sufficiently large, every component
$\Delta_n $ of $M_n\cap \B(1)$, and its closure $\ov{\Delta}_n $, is a disk that is contained
 in $\cC_n$. We first prove that $\Delta_n$ is contained in $\cC_n$. Once this is proved,
the fact that $\Delta_n$ and its closure are disks with piecewise smooth
boundary follows by using the same  arguments as in the proof of item A of the theorem.

Arguing by contradiction, suppose that after choosing a subsequence,
for all $n\in \N$, there is a component
$\Delta_n $ of $M_n\cap \B(1)$ that is not contained in $\cC_n$.
Let $\widehat \Delta_n$ be the connected component of $M_n\cap \B(1+\frac12\be)$ containing $\Delta_n$.
By elementary separation property, using that $\cC_n$ is a disk disjoint
from $\Delta_n$ and with $\partial \cC_n \subset [\rth -\oB(1+\frac23 \be)]$,
it follows that $\widehat \Delta_n$ is also disjoint from $\cC_n$.
By the construction of $\cC_n$, Claim~\ref{two-components} implies that
given $\t_0\in [0,1)$, if $n$ is sufficiently large
then the set $(\bigcup_{\t\in [-\t_0,\t_0]}\g_n^\t) \cap \oB(1+\frac12\be)$
is contained in $ \cC_n$.
Since  $\widehat \Delta_n$ is disjoint from $\cC_n$, we conclude that
given $\ve>0$ there exists $N(\ve)\in\mathbb N$ such that if $n\geq N(\ve)$,
for any point $p\in\widehat \Delta_n$, $|\langle N(p), (0,0,1)\rangle|\geq 1-\ve$.
Since  $\widehat \Delta_n$ contains points in $\B(1)$ and has its boundary contained
in $ \partial \B(1+\frac12\beta)$, if $\ve$ is chosen sufficiently small
depending on $\be$,  then $\widehat\Delta_n$ must
intersect $\g(n)$ when $n$ is sufficiently large, contradicting embeddedness.
 This contradiction
 proves that the disk $\Delta_n$ must be contained in $\cC_n$, and completes the
 proof of the theorem in the special case that $\G_n$ has one component.


Suppose now that $\G_n$ has exactly two components for all $n$.
Let $\G_n\cap \oB(1+\be)=\g_1(n)\cup\g_2(n)$, with $\vec 0\in \g_1(n)$.
Since the arcs $\g_2(n)$
converge to the intersection of the $x_3$-axis with the ball
$\B(1+\be)$, after replacing by a subsequence,
$\g_2(n)\cap \B(\frac1n)\neq\mbox{\rm \O}$. As in the
just considered case, we can construct a
disk $\cC_n^\be(1)$ passing through the origin
and that contains  the ``axis" $\g_1(n)$ and the intrinsic diameter of $\cC_n^\be(1)$
is less than 3.
After choosing a possibly smaller
value of $\be$, one can also construct a
similarly defined ``ruled'' disk $\cC_n^\be(2)\subset \oB(1+2\be)$
with ``axis" $\g_2 (n)$
and the intrinsic diameter of  ${\cC}_n^\be(2)$ is also less than 3.
Note that in this second case the
disks ${\cC}_n^\be$(1), ${\cC}_n^\be(2)$ are also pairwise disjoint.
The proof that every component of $M_n\cap \B(1)$ is a disk whose
closure in $M_n$ is a piecewise smooth compact disk contained in
$\cC_n={\cC}_n^\be(1) \cup{\cC}_n^\be(2)$ is the same as the proof when
$\G_n$ is a single curve.
This last observation completes the proof of Theorem~\ref{thm2.1}
\end{proof}

\section{Applications of the main theorem}

\subsection{Chord-arc Property}
In this section we prove the chord arc result stated in
Theorem~\ref{main2}. Before doing so, we
need to recall a result from~\cite{mt8}. In~\cite{mt8}, we
applied the one-sided curvature estimate in Theorem~\ref{th}
to prove a relation between intrinsic
and extrinsic distances in an $H$-disk, which can be viewed as a
{\em weak chord arc} property.
This result was motivated by and generalizes a previous result,
Proposition~1.1 in~\cite{cm35}, by
Colding-Minicozzi for 0-disks.
 We begin by making the following definition.

\begin{definition} Given a point $p$ on a surface
$\Sigma\subset \rth$, $\S (p,R)$ denotes the closure of
 the component of $\Sigma \cap {\B}(p,R)$ passing through $p$.
\end{definition}

\begin{theorem}[Weak chord arc property, Theorem 1.2 in~\cite{mt8}] \label{thm1.1}
There exists a $\delta_1 \in (0,
\frac{1}{2})$  such that the following holds.

Let $\S$ be an   $H$-disk in $\rth.$  Then for all
intrinsic closed balls $\ov{B}_\S(x,R)$ in $\S-
\partial \S$:

\ben \item $\S (x,\delta_1 R)$ is a disk with piecewise smooth boundary
$\partial \Sigma(\vec{0},\delta_1 R)\subset \partial \B(\de_1R)$. \item
$
 \S (x, \delta_1 R) \subset B_\S (x, \frac{R}{2}).$
\een
\end{theorem}

 Theorem~\ref{main2}
is   equation~\eqref{eqChord2}  of this  result.

\begin{theorem} \label{main3}
Given $\ve>0$, there exists a $C_\ve>1$ so that the following holds.
Suppose that $\S$ is an $H$-disk, $ \vec{0}\in \S$ and $R>r_0>0$.
If  ${B}_\S(\vec 0,C_\ve R)\subset \S-\partial \S$ and
$\sup_{B_\S(\vec 0,(1-\frac{\sqrt2}2)r_0)}|A_\S|\geq r_0^{-1}$,
then
\begin{equation}  \label{eqChord}
\mbox{\rm dist}_\S (x,\vec{0})< (1+\ve)({\sqrt{2}}\,|x|+2 r_0), \;
\mbox{\rm for } x\in B_\S(\vec 0,R).
\end{equation}
In particular, for $\ve$ chosen sufficiently small, there exists a $C_\ve$  so that
\begin{equation}  \label{eqChord2} \frac13 \mbox{\rm dist}_\S (x,\vec{0})< |x|/2+r_0, \;
\mbox{\rm for } x\in B_\S(\vec 0,R).\end{equation}
\end{theorem}

\begin{remark} \label{sharp}
The constants in equation~\eqref{eqChord} are sharp  in the following sense.
For every $k\in \N$,  there exists
a helicoid $\S$ with $\vec{0}\in \S$ satisfying the following properties: \bit 
\item
$\sup_{B_{\S}(\vec 0,1-\frac{\sqrt2}2)}|A_{\S}|= 1,$
\item  there exists  points
$ x,y\in\S$, such that
$$ 
\mbox{\rm dist}_\S (x,\vec{0})> (\sqrt{2}-\frac1k)|x|,$$ $$
\mbox{\rm dist}_\S (y,\vec{0})> 2-\frac1k .
$$
\eit

\end{remark}

\begin{proof}
Arguing by contradiction, suppose there exists $\delta>0$ such that
for any $n\in\mathbb N$ the following holds. There exist
an $H_n$-disk $\S_n$ with $ \vec{0}\in \S_n$, and $R_n>r_n>0$,
such that ${B}_{\S_n}(\vec 0,n^2R_n)\subset \S_n-\partial \S_n$,
$\sup_{B_{\S_n}(\vec 0,(1-\frac{\sqrt2}2)r_n)}|A_{\S_n}|\geq r_n^{-1}$, and

\[
\mbox{\rm dist}_{\S_n} (x_n,\vec{0})\geq (1+\delta)({\sqrt{2}}\,|x_n|+2 r_n)
\]
for some  $x_n\in B_{\S_n}(\vec 0,R_n).$

Abusing the notation, let $ \S_n$ denote the rescaled surfaces
$\frac{1}{nR_n} \S_n$ and let ${r}_n$ denote $\frac{{r}_n}{nR_n}$.
Then  for the new sequence of disks $\S_n$ the following holds.
The intrinsic ball $  B_{ \S_n}(\vec 0,n)$  is
contained in $ \S_n-\partial \S_n$,  $\vec{0}\in  \S_n$,
\begin{equation}\label{blowup}
\sup_{B_{ \S_n}(\vec 0, (1-\frac{\sqrt2}2){r}_n) }|A_{ \S_n}|\geq{r}_n^{-1},
\end{equation}
where $ r_n<\frac 1n$ and

\begin{equation}\label{contra2}
\mbox{\rm dist}_{\S_n} (y_n,\vec{0})\geq (1+\delta)({\sqrt{2}}\,|y_n|+2 r_n)
\end{equation}
 for some $y_n\in B_{\S_n}(\vec 0,1/n).$

Theorem~\ref{thm1.1} implies
that the surfaces $\wt\S_n= \S_n(\vec 0,\de_1 n)$
are disks with piecewise smooth boundary
$\partial \wt \S_n\subset \partial \B(\de_1n)$ and
therefore,
after choosing a subsequence, we can assume
that
$\partial\wt \S_n \subset [\rth - \B(n)]$.

Since by equation~\eqref{blowup},
$\sup_{B_{\wt \S_n}(\vec 0,(1-\frac{\sqrt2}2)r_n) }|A_{ \wt \S_n}|\geq r_n^{-1}\geq n,$
we can apply  the arguments in the proof of Case B of
Theorem~\ref{thm2.1} to the sequence of constant mean curvature
disks $\wt \S_n$. Once again, abusing the notation, we let $\S_n$ denote $\wt \S_n$.
Theorem~\ref{thm2.1} implies that,
after choosing a subsequence, there exists a straight
line $\cS$ going through the origin, such that the sequence of surfaces $\S_n$
converges to a minimal foliation of
$\rth-\cS$ by planes. Without loss of generality, we will assume these
planes are horizontal and $\cS$ is the $x_3$-axis.

By the proof of part B of Theorem~\ref{thm2.1},
there is a set $\G_n\subset \S_n\cap \B(3)$ that consists of
one or two analytic arcs along which the disk $\S_n$ has vertical
tangent planes. These curves converge $C^1$ to the
intersection of the $x_3$-axis with $ \B(3)$. Finally there exists
a disk $\cD^\be_n$ component of a set $\cC_n^\be$ that
contains the intrinsic ball
$\ov{B}_{\S_n}(\vec{0},1/n)$ and contains a component
$\g(n) $ of $\G_n\cap \cC_n^\be \subset \B(1+2\be)$, $\be \in (1,\frac13)$. Recall that
by taking $n$ sufficiently large, $\beta$ can be taken arbitrarily small.

In order for it to be easier to
apply the arguments in the proof  of part B of Theorem~\ref{thm2.1},
we will soon translate  the surfaces $\S_n$ so that the translations of
the curves $\g(n)$ contain  the origin.

First recall from the proof of  part B of Theorem~\ref{thm2.1}
that we can
view $\cD^\be_n$ to be ruled by curves $\a_t^n$, each of which intersects
$\g(n)$ in a single point with $x_3$-coordinate equal to $t$.
Let $p_n\in \g(n)$ be the point such that $\vec 0\in \a_{x_3(p_n)}^n$
and consider the translated surfaces $\wt\Sigma_n=\Sigma_n-p_n$ so that $p_n$
is sent to the origin, the origin is sent to $q_n:=-p_n$ and $y_n$ is sent
to $z_n:=y_n-p_n$. Abusing the notation, we use the same notation for the
translation of $\g(n)$ and $\cD^\be_n$. Let $G_n^\be$ be the map as
defined in equation~\eqref{mapg} of the previous section, but restricted
to $\cD^\be_n\cap \B(\frac 12)$. Note that
since $q_n\in \a_0^n$, $x_3(G_n^\be(q_n))=0$. Let $Q_n=G_n^\be(q_n)$ and $Z_n=G_n^\be(z_n)$.
In this new setting, equation~\eqref{contra2} becomes
\begin{equation}\label{contra3}
 \mbox{\rm dist}_{\wt{\S}_n} (z_n,q_n)\geq (1+\delta)({\sqrt{2}} |z_n-q_n|+2r_n).
\end{equation}
We will obtain a contradiction by estimating
$\mbox{\rm dist}_{\wt{\S}_n} (z_n,q_n)$ from above.

 Recall from the previous
section the definitions of $c_{z_n,q_n}^{\be,n}$ and $C_{z_n,q_n}^{\be,n}$,
that
\[
\mbox{\rm dist}_{\wt{\S}_n} (z_n,q_n)\leq \mbox{\rm Length}(c_{z_n,q_n}^{\be,n})\leq (1+\be)^2\mbox{\rm Length}(C_{z_n,q_n}^{\be,n}),
\]
  and that by taking $n$ sufficiently large, $\beta$ can be taken arbitrarily small.
This, together with equation~\eqref{contra3} gives that
\begin{equation*}
(1+\delta)({\sqrt{2}} |z_n-q_n|+2r_n)\leq (1+\beta)^2\mbox{\rm Length}(C_{z_n,q_n}^{\be,n})\leq (1+\delta)^\frac12\mbox{\rm Length}(C_{z_n,q_n}^{\be,n}),
\end{equation*}
if $(1+\beta)^2\leq (1+\delta)^\frac12$. In other words,
\begin{equation}\label{boundlength}
(1+\delta)^\frac12({\sqrt{2}} |z_n-q_n|+2r_n)\leq \mbox{\rm Length}(C_{z_n,q_n}^{\be,n}),
\end{equation}

By working with the distance $D$ in Definition~\ref{d}, it can be shown that if
there exists $t>0$ such that $\mbox{\rm Length}(C_{z_n,q_n}^{\be,n})=t+2|Q_n| $,
  then $|Q_n-Z_n|\geq \frac{\sqrt 2}2 t$,
see Figure~\ref{Dball}. Note that since $x_3(Q_n)=0$, then $|Q_n|$ is equal
to the distance from $Q_n$ to the $x_3$-axis.
In order to prove that such a positive $t$ exists, it will suffice to show that
for $n$ large, $\mbox{\rm Length}(C_{z_n,q_n}^{\be,n})-2|Q_n| >0$. For the
time being, let us assume that such a positive $t$ exists. Then, by the previous discussion, $t\leq \sqrt 2|Q_n-Z_n|$ and thus
\[
\mbox{\rm Length}(C_{z_n,q_n}^{\be,n})=t+2|Q_n|  \leq \sqrt 2 |Q_n-Z_n|+2|Q_n|.
\]
This, together with inequality~\eqref{boundlength}, implies that
\[
 (1+\de)^\frac12({\sqrt{2}}\,|z_n-q_n|+2 r_n)\leq \sqrt 2 |Q_n-Z_n|+2|Q_n|.
\]


\begin{claim}\label{distance}
Given $\mu>0$ there exists $N(\mu)>0$ such that for any $n\geq N(\mu)$, $|Q_n|\leq  (1+ \mu)r_n$.
\end{claim}

\begin{proof}
 Arguing by contradiction, assume that there exists $\mu>0$ such that
 $  |Q_n|\geq (1+\mu)  r_n$. By the arguments in Section~\ref{hlam},
\begin{equation} \label{EQ13}
 \mbox{\rm dist}_{\wt{\S}_n} (\vec 0,q_n)\leq (1+\beta)|Q_n|\leq 2|Q_n|\;\text{
and } \sup_{B_{\wt\S_n}( q_n,(1-\frac{\sqrt2}2)r_n)}|A_{\wt\S_n}|\geq r_n^{-1}.
\end{equation}
Consider the sequence of rescaled surfaces $\S_n'$ given by
$\frac{1}{ |Q_n|}\wt \S_n$ and let $q_n'=\frac{1}{ |Q_n| }q_n$, $Q_n'=\frac{1}{ |Q_n| }Q_n$. Note that
$\mbox{\rm dist}_{ {\S}'_n} (\vec 0,q'_n)\leq 2$.
Since $Q_n'$ is the projection of $q_n'$ to the $(x_1,x_2)$-plane, the
 Euclidean distance from $q'_n$ to the $x_3$-axis is  equal to $1$ while the
 distance from $q'_n$ to the $(x_1,x_2)$-plane is bounded.  Note also that
 \[
 \sup_{B_{\S'_n}( q'_n,(1-\frac{\sqrt2}2)\frac{r_n}{|Q_n|})}|A_{\S'_n}|\geq \frac{|Q_n|}{r_n}\geq 1+\mu
 \] and the tangent
 plane to $\S_n'$ at $\vec 0$ is vertical.

By Theorem~\ref{thm2.1},
see also the arguments in Claim~\ref{twocomponents}, after going to a subsequence, $\S_n'$   converges either to a
vertical helicoid containing the $x_3$-axis or to a foliation of $\rth$ minus the
$x_3$-axis by horizontal planes. Since  $\sup_{B_{\S'_n}( q'_n,(1-\frac{\sqrt2}2)\frac{r_n}{|Q_n|})}|A_{\S'_n}|\geq 1+\mu$,
 $\frac{r_n}{|Q_n|}\leq \frac1{1+\mu}<1$, and the
 Euclidean distance from $q'_n$ to the $x_3$-axis is  equal to $1$, $\S_n'$ must converge  to a vertical helicoid $\cH$ containing the $x_3$-axis. Let
 \[
 \lim_{n\to\infty}  \frac{|Q_n|}{r_n}=: T\in[1+\mu, \infty).
 \]

By the previous discussion and the definitions of $q'_n$ and $Q_n'$, $\lim_{n\to\infty} q'_n$ exists and it is equal to $Q:=\lim_{n\to\infty} Q'_n$.
Note that
\[
\sup_{B_{\cH}  (Q,(1-\frac{\sqrt2}2)\slash T)}|A_{\cH}|\geq T\geq 1+\mu,
\]
$|Q|=1$ and $x_3(Q)=0$. Thus there exists a point $P\in \ov{B}_{\cH}(Q,(1-\frac{\sqrt2}2)\slash T)$ with
$|A_\cH|(P)\geq T>1$ with its distance from the $x_3$-axis at least $1-(1-\frac{\sqrt2}2)\slash T>  \frac{\sqrt2}2$.
This contradicts the geometric property that on  a vertical helicoid containing the $x_3$-axis, points of distance greater than $\frac{\sqrt2}2$ from the $x_3$-axis must have norm of the second fundamental form less than 1. This  finishes the proof of the claim. \end{proof}

By taking $\mu$ such that $1+\mu \leq (1+\de)^\frac12-\rho$, for a fixed $\rho\in (0,(1+\de)^\frac12-1)$ we have obtained that for $n$ sufficiently large
\[
 (1+\de)^\frac12({\sqrt{2}}\,|z_n-q_n|+2 r_n)\leq \sqrt 2 |Q_n-Z_n|+2|Q_n|
\]
\[
\leq \sqrt 2 |Q_n-Z_n|+2((1+\de)^\frac12-\rho)r_n.
\]
In other words,
\begin{equation} \label{EQ14}
 (1+\de)^\frac12{\sqrt{2}}\,|z_n-q_n|\leq \sqrt 2 |Q_n-Z_n|-2\rho r_n.
\end{equation}

Next we use Claim~\ref{distance} to prove that
$\mbox{\rm Length}(C_{z_n,q_n}^{\be,n})-2|Q_n| >0$. By equation~\eqref{boundlength},
\[
(1+\delta)^\frac12 2r_n\leq  \mbox{\rm Length}(C_{z_n,q_n}^{\be,n}).
\]
Therefore, using this, Claim~\ref{distance} and our choice of $\mu$ we have that if $n$ is sufficiently large,
\[
\mbox{\rm Length}(C_{z_n,q_n}^{\be,n})-2|Q_n| \geq (1+\delta)^\frac12 2r_n-(1+\mu)2r_n\geq 2\rho r_n>0.
\]

In order to finish the proof, we argue similarly to the proof of Claim~\ref{distance}.
After passing to a subsequence, we can assume that one of the following three cases holds.

\begin{enumerate}
\item $\lim_{n\to\infty}\frac{r_n}{|z_n|}=L\in (0, \infty)$;
\item $\lim_{n\to\infty}\frac{r_n}{|z_n|}=0$;
\item $\lim_{n\to\infty}\frac{r_n}{|z_n|}=\infty$.
\end{enumerate}

First consider the case  where  $\lim_{n\to\infty}\frac{r_n}{|z_n|}=L\in (0, \infty)$.
Consider the sequence of rescaled surfaces
$\S_n'$ given by
$\frac{1}{ |z_n|}\wt \S_n$ and let $q_n'=\frac{1}{ |z_n| }q_n$,
$Q_n'=\frac{1}{ |z_n| }Q_n$, $z_n'=\frac{1}{ |z_n| }z_n$, $Z_n'=\frac{1}{ |z_n| }{Z_n}$.

By equation~\eqref{EQ13}, equation~\eqref{EQ14} and Claim~\ref{distance}, we then have
\[
d_{\S_n'}(q_n',\vec 0)\leq (1+\beta)|Q_n'|\leq (1+\beta)(1+\mu)\frac{r_n}{|z_n|},
\]
\[
\sup_{B_{\S'_n}( q'_n,(1-\frac{\sqrt2}2)\frac{r_n}{|z_n|})}|A_{\S'_n}|\geq \frac{|z_n|}{r_n},
\]
\[
 (1+\de)^\frac12{\sqrt{2}}\,|z'_n-q'_n|\leq \sqrt 2 |Q'_n-Z'_n|-2\rho \frac{r_n}{|z_n|}.
\]
Notice that after replacing by a subsequence,  the sequence of  points $z'_n$ converges to a point on the unit sphere.

 By Theorem~\ref{thm2.1},
see also the arguments in Claim~\ref{twocomponents}, and the fact that $\lim_{n\to\infty}\frac{|z_n|}{r_n}=\frac1L>0$, a subsequence of $\S_n'$   converges either to a
vertical helicoid containing the $x_3$-axis or to a foliation of $\rth$ minus the
$x_3$-axis by horizontal planes. (Note that when $\lim_{n\to\infty}\frac{r_n}{|z_n|}=0$, in which case $\lim_{n\to\infty}\frac{|z_n|}{r_n}=\infty$,
the surfaces $\S_n'$   converge  to a foliation of $\rth$ minus the
$x_3$-axis by horizontal planes.)

In either case, $\lim_{n\to\infty} z_n'$ exists and it is equal to $\lim_{n\to\infty} Z_n'$. Similarly, $\lim_{n\to\infty} q_n'$ exists and it is equal to $\lim_{n\to\infty} Q_n'$. In particular, $\lim_{n\to\infty} |z'_n-q'_n|$ exists and it is equal to $\lim_{n\to\infty} |Z'_n-Q'_n|$. Thus, since
\[
\lim_{n\to\infty}2\rho \frac{r_n}{|z_n|}=2\rho L>0,
\]
by choosing $n$ sufficiently large, one   contradicts the inequality
\[
 (1+\de)^\frac12{\sqrt{2}}\,|z'_n-q'_n|\leq \sqrt 2 |Q'_n-Z'_n|-2\rho \frac{r_n}{|z_n|}.
\]

Suppose that $\lim_{n\to\infty}\frac{r_n}{|z_n|}=0$. This being the case, in order to obtain a contradiction, we argue exactly as in the previous case and note that $\lim_{n\to\infty} |z'_n-q'_n|=1\neq 0$ since $\lim_{n\to\infty} q'_n=\vec 0$ and $z_n'$ converges to a point on the unit sphere. Thus, by choosing $n$ sufficiently large, one again contradicts the inequality
\[
 (1+\de)^\frac12{\sqrt{2}}\,|z'_n-q'_n|\leq \sqrt 2 |Q'_n-Z'_n|-2\rho \frac{r_n}{|z_n|}.
\]

It remains to obtain a contradiction when  $\lim_{n\to\infty}\frac{r_n}{|z_n|}= \infty$. In this case, consider the sequence of rescaled surfaces
$\frac{1}{ r_n}\wt \S_n$. Abusing the notation, let $\S_n'$ denote this sequence and let $q_n'=\frac{1}{ r_n }q_n$, $Q_n'=\frac{1}{ r_n }Q_n$, $z_n'=\frac{1}{ r_n }z_n$, $Z_n'=\frac{1}{ r_n }{Z_n}$.   By equation~\eqref{EQ13}, equation~\eqref{EQ14} and Claim~\ref{distance}, we then have
\[
d_{\S_n'}(q_n',\vec 0)\leq (1+\beta)|Q_n'|\leq (1+\beta)(1+\mu),
\]
\[
\sup_{B_{\S'_n}( q'_n,1-\frac{\sqrt2}2)}|A_{\S'_n}|\geq 1,
\]
\[
 (1+\de)^\frac12{\sqrt{2}}\,|z'_n-q'_n|\leq \sqrt 2 |Q'_n-Z'_n|-2\rho.
\]

Arguing like in the previous case, by Theorem~\ref{thm2.1} a subsequence of the
surfaces $\S_n'$  converges either to a
vertical helicoid containing the $x_3$-axis or to a foliation of $\rth$ minus the
$x_3$-axis by horizontal planes. In either case, since
$\lim_{n\to\infty}\frac{r_n}{|z_n|}= \infty$, then $|z_n'|=\frac{|z_n|}{r_n}$
is bounded from above, and it follows that $\lim_{n\to\infty} z_n'$ exists and it is
equal to $\lim_{n\to\infty} Z_n'$. Similarly, $\lim_{n\to\infty} q_n'$ exists and it
is equal to $\lim_{n\to\infty} Q_n'$. In particular $\lim_{n\to\infty} |z'_n-q'_n|$
exists and it is equal to $\lim_{n\to\infty} |Z_n'-Q_n'|$. Therefore, similarly to the previous two cases,
by choosing $n$ sufficiently large, one contradicts the inequality
\[
 (1+\de)^\frac12{\sqrt{2}}\,|z'_n-q'_n|\leq \sqrt 2 |Q'_n-Z'_n|-2\rho.
\]

This finishes the proof of the theorem.\end{proof}

\subsection{Curvature Estimates.} \label{sec:appl}

In this section we   prove three other useful
corollaries of Theorem~\ref{th} and Corollary~\ref{cest2}. The
next corollary essentially states that if three $H$-disks are
sufficiently close to a point that is away from their boundaries,
then nearby that point, the surfaces
satisfy a curvature estimate.

\begin{corollary} \label{cest} There exist constants $\ve<1$, $C>1$
such that the following holds.  Let
$\Sigma_1$, $\Sigma_2$, $\Sigma_3$ be three pairwise disjoint $H_i$-disks with
$\partial \Sigma_i\subset [ \rth- \B(1)]$ \,for $i=1,2,3$.
If $\,\B(\ve)\cap\Sigma_i\not=\O$ for $i=1,2,3$, then {\large
\[
\sup_{\B(\ve)\cap\Sigma_i,\,i=1,2,3}|A_{\Sigma_i}|\leq C.
\]}\end{corollary}

\begin{proof}
Arguing by contradiction, suppose that the corollary fails.
In this case, for $i=1,2,3$, there exist
sequences of  $H_i(n)$-disks $\Sigma_i(n)$ with
$\partial \Sigma_i(n)\subset \partial \B(1)$ such
that $\B(\frac{1}{n^2} )\cap\Sigma_i\not=\O$ but {\large
\[
\sup_{\B(\frac{1}{n^2} )\cap[\cup_{i=1}^3\Sigma_i(n)]}|A_{\cup_{i=1}^3\Sigma_i(n)}|> n^2.
\]}

Consider the sequence of rescaled surfaces $M_i(n)=n\Sigma_i(n)$,
$i=1,2,3$, and note that $M_i(n)\cap \B(1/n)\neq\O$ for each $i$.
Without loss of generality, we can assume that {
\[
\sup_{\B(1/n)\cap M_1(n)}|A_{M_1(n)}|> n.
\]}
After replacing by a subsequence, Theorem~\ref{thm2.1} implies that
the sequence $M_1(n)$ converges
$C^\alpha$, $\a\in (0,1)$,  to a foliation $\cF$ of $\rth-\cS_1$ by
parallel  planes, where $\cS_1$
is a  line orthogonal to
the planes; without loss of generality we will henceforth
assume that the parallel planes are horizontal.

Since for $i=2,3$, $M_i(n)\cap\B(1/n)\neq\O$,
Theorem~\ref{thm2.1},  Corollary~\ref{cest2}
and the embeddedness of the disconnected surface $\cup_{i=1}^3M_i(n)$
imply that, after replacing by  new
subsequences, each of the sequences $M_i(n)$  converges
$C^\alpha$, for $\a\in (0,1)$, to the foliation $\cF$
of $\rth-\cS_1$ by horizontal planes.
It follows that for $n $ large, each
of the surfaces $M_i$, $i=1,2,3$,  yields at least one related analytic
arc component in the set $\G_n$ where
$\cup_{i=1}^3M_i(n)$ has vertical tangent planes in $\B(3)$.
But the total number of such components in $\G_n$
must be at most two by adapting the arguments
in the proof of Claim~\ref{twocomponents}.
This gives a contradiction which completes the proof of the claim.
\end{proof}

From Corollary~\ref{cest} one easily obtains the result below.

\begin{corollary} \label{cor:3} There exists $\ve\in (0,1)$ such that the
following holds.  Let $\Sigma_1$, $\Sigma_2$, $\Sigma_3$ be
$H_i$-disks with $\partial \Sigma_i\subset \partial \B(1)$ for
$i=1,2,3$. If $\B(\ve)\cap\Sigma_i\not=\O$ for
$i=1,2,3$, then, after a rotation, each component of
$\Sigma_i\cap \B(2\ve)$ intersecting $\B(\ve)$ is a graph
with norm of its gradient bounded by 1 over its projection into the plane $\{x_3=0\}.$
\end{corollary}

\begin{remark}{\rm
In Corollary~0.4 of~\cite{cm23}, Colding and Minicozzi gave an
analogous result to the statement of  Corollary~\ref{cest}.
In their case, it suffices to have two embedded
minimal disks to obtain a curvature estimate.
This is because one can insert a stable minimal disk,
satisfying curvature estimates, in between
the two minimal disks and apply a
result analogous to Corollary~\ref{cest2} for minimal disks. In our case this approach does
not work and indeed the curvature estimate is not true for just two $H$-disks.}
\end{remark}

In the next proposition, we use the approach described in the previous
remark to obtain a curvature estimate when certain topological conditions are satisfied.

\begin{corollary}  There exist constants $\ve<\frac12$, $C>1$
such that the following holds.  Let
$\Sigma_1$, $\Sigma_2$ be two pairwise disjoint $H_i$-disks with
$\partial \Sigma_i\subset [ \rth- \B(1)]$ \,  for $i=1,2$, and
$0\leq H_1\leq H_2$. Let $W$ denote a component of $\B(1)-[\Sigma_1\cup\Sigma_2]$ such
that the mean curvature vector of $\S_2$ points into $W$.
If $W\cap\B(\ve)$  contains a component $W'$  such that
$\Sigma_i\cap\partial W'\neq \O$ for $i=1,2$  , then {\large
\[
\sup_{\B(\ve)\cap\Sigma_i,\,i=1,2}|A_{\Sigma_i}|\leq C.
\]}\end{corollary}

\begin{proof}
By Corollary~\ref{cest2}, there exists $\ve>0$ such that the following holds:
if $M\subset W$ is a surface such that
$M\cap\B(\ve)\neq\O$, $\partial M\subset \partial \B(1)$ and $M$ satisfies
a uniform estimate on the norm of its second fundamental form
in $\B(1\slash 2)$ that does not depend on $\Sigma_1$ or $\Sigma_2$, then the
corollary holds. Therefore, it suffice to construct such an $M$. If $H_2\geq 1$,
then by the extrinsic curvature estimates in Theorem~\ref{ext:cest}, $\Sigma_2$
can be chosen to be such $M$. Therefore, we can assume that $H_2<1$.

Under our hypothesis, there exists a compact, oriented, weakly
stable $H_1$-surface $M_1\subset W$ such that
$\partial M_1= \partial \Sigma_1\cap\partial W$ and
$M_1$ is homologous to $\Sigma_1\cap\partial W$ in $W$,
see for instance~\cite{alr1} for this construction.
By the results in~\cite{lor2} and standard rescaling
arguments, $M_1$ satisfies a uniform estimate on the
norm of the second fundamental form in $\B(1\slash 2)$.
Since  $M_1$ is homologous to $\Sigma_1\cap\partial W$
in $W$  and $W\cap\B(\ve)$  contains a component $W'$  such that
$\Sigma_i\cap\partial W'\neq \O$ for $i=1,2$,
then $M_1\cap\B(\ve)\neq\O$. This finishes the proof of the
corollary.
\end{proof}

\vspace{.3cm}
\center{William H. Meeks, III at profmeeks@gmail.com\\
Mathematics Department, University of Massachusetts, Amherst, MA 01003}
\center{Giuseppe Tinaglia at giuseppe.tinaglia@kcl.ac.uk\\ Department of
Mathematics, King's College London,
London, WC2R 2LS, U.K.}

\bibliographystyle{plain}
\bibliography{bill}

\end{document}